\documentclass[final]{siamltex}
\usepackage{graphicx}
\usepackage{amsmath,amssymb,amsfonts,bm}    
\usepackage{graphics}                           
\usepackage{color}                              
\usepackage{dsfont}
\usepackage{hyperref}                           
\usepackage{multirow}                           
\usepackage{graphicx}                           
\usepackage{amscd}                              
\usepackage{mathrsfs}
\usepackage{makeidx}
\usepackage{subfigure}                          
\usepackage{srcltx}                             %
\usepackage{fancyhdr}
\usepackage{float}
\usepackage{pgf}
\usepackage{tikz}
\usetikzlibrary{decorations.pathmorphing}
\usetikzlibrary{fadings}
\usepackage{pgfkeys} 
\usepackage{ifthen}
\usepackage{xargs,cleveref,enumerate}

\usepackage[normalem]{ulem} 

\DeclareMathOperator*{\diver}{\mathrm{div}}

\newcommandx{\yaHelper}[2][1=\empty]{%
\ifthenelse{\equal{#1}{\empty}}%
  { \ensuremath{ \scriptstyle{ #2 } } } 
  { \raisebox{ #1 }[0pt][0pt]{ \ensuremath{ \scriptstyle{ #2 } } } }  
}   

\newcommandx{\yrightarrow}[4][1=\empty, 2=\empty, 4=\empty, usedefault=@]{%
  \ifthenelse{\equal{#2}{\empty}}
  { \xrightarrow{ \protect{ \yaHelper[ #4 ]{ #3 } } } } 
  { \xrightarrow[ \protect{ \yaHelper[ #2 ]{ #1 } } ]{ \protect{ \yaHelper[ #4 ]{ #3 } } } } 
}

\newcommand{\Mosco}{\hspace{-.05cm}\yrightarrow{\scriptscriptstyle \mathrm{M}}[-1pt]\hspace{-.05cm}}

\newtheorem{thm}{Theorem}

\newtheorem{cor*}{Corollary}

\newtheorem{ex}{Example}[section]

\newtheorem{assumption}{Assumption}
 
 \newcommand{\dif}{\:\mathrm{d}}
\newcommand{\diam}{\:\mathrm{diam}}
\newcommand{\K}{\mathbf{K}}
\newcommand{\R}{\mathbb{R}}
\newcommand{\m}{\mathcal}
\newcommand{\fa}{\forall\:}
\newcommand{\hdiv}{H(\Omega;\mathrm{div})}
\newcommand{\qed}{\;\;\ensuremath{\endproof}}

\title{On Some Quasi-Variational Inequalities and other problems with moving sets \thanks{C. N. R. was supported by NSF grant DMS-2012391, and  acknowledges the support of Germany's Excellence Strategy - The Berlin Mathematics Research Center MATH+ (EXC-2046/1, project ID: 390685689) within project AA4-3.}  
}
 
\author{  Jos\'{e}-Luis Menaldi\thanks{Department of Mathematics, Wayne State University, Detroit, MI 48202, USA. ({\tt menaldi@wayne.edu})} \and  Carlos N. Rautenberg\thanks{Department of Mathematical Sciences and the Center for Mathematics and Artificial Intelligence (CMAI), George Mason University, Fairfax, VA 22030, USA. ({\tt crautenb@gmu.edu})
}
    }

\begin{document}

\maketitle
\begin{abstract}
Since its introduction over 50 years ago, the concept of Mosco convergence has permeated through diverse areas of mathematics and applied sciences. These include applied analysis, the theory of partial differential equations, numerical analysis, and infinite dimensional constrained optimization, among others. In this paper we explore some of the consequences of Mosco convergence on applied problems that involve moving sets, with some historical accounts, and modern trends and features. In particular, we focus on connections with density of convex intersections, finite element approximations, quasi-variational inequalities, and impulse problems.

\end{abstract}
\begin{center}
\textbf{On the occasion of the 80th birthday of Umberto Mosco}
\end{center}



\pagestyle{myheadings}
\thispagestyle{plain}

\section{Introduction}
The overwhelming success of Mosco convergence \cite{mosco1969convergence,mosco1967approximation} is present in several areas of mathematics. The concept provided the right framework for the study of problems involving \emph{moving} convex sets in reflexive Banach spaces. In fact, it made possible to study perturbation/stability properties for variational inequalities and other nonlinear problems in calculus of variations, provided existence results for quasi-variational inequalities (QVIs), and allowed the study of optimization problems where control/design variables modify constraints. 

Historically, set convergence notions in abstract spaces go back to Painlev\'{e} in the beginning of the nineteenth century. They were, however, popularized by Kuratowski \cite{kuratowski1948topologie} who remains as part of the name in the set limit names.  In between the appearance of the latter, and the one of the seminal {paper by} Di Giorgi and Franzoni \cite{de1975tipo} on $\Gamma$-convergence (a full analysis of relationships between $K-$ and $\Gamma-$limits can be found in the monograph {by} Dal Maso \cite{d93}), Mosco published his foundational results \cite{mosco1969convergence,mosco1967approximation} associated to the convergence of sets.  The initial results of Mosco are almost simultaneous to the famous Lions and Stampacchia paper \emph{Variational Inequalities} \cite{Lions1967}. This is not coincidence, as in Mosco's own words, it was Stampacchia ``who suggested this research''. Further, we direct the interested reader to Aubin and Frankowska \cite{Aubin2009} for an historical account on the notions of set convergence.

There is a vast literature on \emph{Impulse Control Problems} and their 
connections with QVIs, the reader may check the books {by} Bensoussan and 
Lions~\cite{Bensoussan1984} and Bensoussan~\cite{Ben1982} {for a self-contained account on the subject}. In fact, QVIs were initially identified and treated by Bensoussan and Lions \cite{Bensoussan1974,Lions1973} through impulse control problems. Hamilton-Jacobi-Bellman equations take the form of QVIs in many applications to (stochastic) control/design problems.  Thus, the convergence of sets (which is relevant in the setting of the problems itself) becomes very important in the approximation and implementation of these models.  
A very short description is given later on, essentially, to alert the reader some other similar (yet different) type of QVIs, where Mosco's convergence has not been completely discussed and {explored}.

From a general perspective, QVIs are nonlinear, nonconvex, and nonsmooth problems with (in general) non-unique solutions. Specifically, a QVI can be seen as variational problem with an implicit (state dependent) constraint. { This leads} to the need to approach these problems from the moving set perspective. This level of complexity established QVIs as powerful physical models. It should be noted that QVIs have been successfully applied to the magnetization of superconductors, Maxwell systems, thermohydraulics, image processing, game theory, surface growth of granular {(cohensionless)} materials, hydrology, and solid and continuum mechanics. {For more details,} we refer the reader to \cite{MR3082853,MR2509561,Duvaut,Harker,Rodrigues2000,Lions1975,Pang,Prigozhin} and {to } the monographs \cite{BaiC1984,kravchuk2007variational}.  

In this paper, in addition to providing an account of some basic sufficient conditions for Mosco convergence for several types of sets, we focus on two main consequential aspects associated to Mosco convergence. Initially, we study its relationship with density properties for convex sets, and provide application {to} finite element discretization of problems involving convex sets. Subsequently, we focus on quasi-variational inequalities, impulse problems, and some stability properties of the set of solutions to the QVI.

{ As the notion of Mosco convergence can also be described via a functional convergence related to $\Gamma$-convergence (this is detailed within the paper),  the concept is directly applicable to the study of regularized minimization problems in highly diverse settings. Some of these that are not fully in the scope of this paper include viscosity solutions of minimization problems  \cite{attouch1996viscosity,attouch1996dynamical}, derivation of variational models for granular material accumulation \cite{igbida2011monge,bocea2012models}, Tikhonoff regularization for inverse problems \cite{neubauer1990finite,neubauer1988tikhonov}, and posteriori error estimates for adaptive finite elements \cite{sv07,feischl2014convergence} (and references within \cite{rosel2017approximation}).}

The paper is organized as follows. In section {\ref{sec:prelim} we provide some common notation used throughout the paper, and the famous definition of Mosco convergence together with basic results involving general classes of convex sets}. A short account on necessary and sufficient conditions of Mosco convergence for unilateral sets are given in section \ref{sec:sufnec}. The role of density properties for convex sets in Mosco convergence is explored in section \ref{sec:density}, and its subsequent application to finite element discretization is provided on section \ref{sec:FE}. Quasi-Variational Inequalities are considered in section \ref{sec:QVIs}, and a short, historical and modern account on impulse control problems is given next on section \ref{sec:impulse}. We finalize the paper with a short account on existence of QVIs and stability results for multivalued problems in sections \ref{sec:existence} and \ref{sec:Order}, respectively.

\section{Notations and Preliminary Results}\label{sec:prelim}

Throughout most of this paper we assume (unless stated otherwise) that $V$ is a reflexive real Banach space of (equivalence) classes of maps of the type $v:\Omega\to \mathbb{R}$ for some Lipschitz domain $\Omega\subset \mathbb{R}^\mathrm{d}$ with $\mathrm{d}\in \mathbb{N}$.  For an arbitrary Banach space $V$  we write $\|\cdot\|_V$ for its associated norm.  
The topological dual is denoted as $V'$, and by $\langle\cdot,\cdot\rangle_{V',V}$ the associated duality pairing.
For a sequence $\{z_n\}_{n\in\mathbb{N}}$ in $V$ we denote its strong convergence to $z\in V$ by ``$z_n\to z$'' and weak convergence by ``$z_n\rightharpoonup z$''.    Further, for two Banach spaces $V_1$ and $V_2$, we write $\mathscr{L}(V_1,V_2)$ for the space of bounded linear operators from $V_1$ to $V_2$.

{ The typical function spaces under consideration are described next. }For an open domain $\Omega$ in $\R^{\mathrm{d}}$, we denote $H^1(\Omega)$ to { be} the Sobolev space of $L^2(\Omega)$ functions whose weak gradients belong to $L^2(\Omega)$, and by $H_0^1(\Omega)$  we denote the subset of $H^1(\Omega)$ whose elements are zero in $\partial \Omega$ in the sense of the trace (provided that $\Omega$ is regular enough). For functions in $L^p(\Omega)$ with gradients in $L^p(\Omega)$, we utilize $W^{1,p}(\Omega)$, and $W_0^{1,p}(\Omega)$ for functions vanishing at $\partial\Omega$.

{ We are now in position to establish the notion of Mosco convergence and some of its basic consequences. }The celebrated definition by Mosco \cite{mosco1969convergence,mosco1967approximation} is the following 

\vspace{.2cm}

\begin{definition}[\textsc{Mosco convergence}]\label{definition:MoscoConvergence}
Let $\K$ and $\K_n$, for each $n\in\mathbb{N}$, be non-empty, closed and convex subsets of $V$. {Then} the sequence \textit{$\{\K_n\}$ {is said to} converge to $\K$ in the sense of Mosco} as $n\rightarrow\infty$, {denoted} by $$\K_n\Mosco\K,$$ if the following two conditions {are fulfilled}:
\begin{enumerate}[\upshape(I)]
  \item\label{itm:1}  For each $w\in \K$, there exists $\{w_{n'}\}$ such that $w_{n'}\in \K_{n'}$ for $n'\in \mathbb{N}'\subset \mathbb{N}$ and $w_{n'}\rightarrow w$ in $V$.
  \item\label{itm:2} If $w_n\in \K_n$ and $w_n\rightharpoonup w$ in $V$ along a subsequence, then $w\in \K$.
\end{enumerate}
\end{definition}

\vspace{.2cm}

In general and in concrete applications, item \eqref{itm:2} in Definition \ref{definition:MoscoConvergence} is significantly simpler to check than \eqref{itm:1}. In fact, \eqref{itm:1} requires clever constructions that leads into problem-tailored approaches.

The relevance of Mosco convergence can be explained by the fact that it provides the right ``topology'' for the obtention of stability results to solutions of variational inequalities when the constraint sets are perturbed. For this matter, consider $\mathbf{K}\subset V$ non-empty, closed and convex and $f\in V'$. We define $S(f,\K)$ as the unique solution to the following variational inequality (VI){:}
\begin{equation}\label{eq:VI}
\text{Find } y\in \mathbf{K}: \langle Ay-f,v-y\rangle \geq 0, \quad \forall v\in \mathbf{K},
\end{equation}
where{ 
\begin{equation}\label{eq:AssA}
	A:V\to V' \text{ is linear, bounded, and strongly monotone},
\end{equation}
and }we assume this about $A$ throughout the paper. { For a detailed account of problem \eqref{eq:VI}, we refer the author to \cite{Kinderlehrer} or \cite{BaiC1984}.} Then, we have the following result by Mosco: If $f_n\to f$ in $V'$, we have that 
\begin{equation*}
\K_n\Mosco\K \quad {\text{implies}}  \quad S(f_n,\K_n)\to S(f, \K) \text{ in } V;
\end{equation*}
see \cite{mosco1969convergence} or \cite{Rodrigues1987}.

In this paper we focus on two classes of problems that share similar difficulties with ``moving'' sets and hence Mosco convergence becomes a crucial tool in their treatment. Initially, we focus on optimization problems and their regularization/discretization and limiting behavior. In particular, we deal with the issue of Mosco convergence via properties of density of convex intersections. Secondly, we focus on some particular classes of quasi-variational inequalities (variational problems with implicit obstacles), and stability properties of the solution set.

We consider a general structure of the sets of interest that is wide enough to include pointwise bounds on function values, their gradient, curl or divergence, and also nonlocal type constraints like the ones arising from linear integral operators. 
The general structure of the sets of interest are of the form 
\begin{equation}\label{eq:KVI}
    \mathbf{K}=\{w\in V: \psi(Gw) \leq \phi \},
\end{equation}
where $\phi:\Omega\to\mathbb{R}$ is a nonnegative measurable function and ``$v\leq w$'' stands for $v(x)\leq w(x)$ for almost all (f.a.a.) {$x\in \Omega$, or almost everywhere (a.e.),} unless stated otherwise. We assume that $G\in \mathscr{L}(V,L^p(\Omega)^{\mathrm{m}})$ for some $1<p<+\infty$ and ${\mathrm{m}}\in \mathbb{N}$, that is, $G\colon V\to L^p(\Omega)^{\mathrm{m}}$ is linear and bounded. Additionally, we suppose that  $\psi\colon \mathbb{R}^{\mathrm{m}}\to\mathbb{R}$ is convex, $\psi(0)=0$, $\psi(tx)=t\psi(x)$ for all $x\in\mathbb{R}^{\mathrm{m}}$ and $t>0$, and it is possibly nonsmooth at the origin but smooth everywhere else. Note that the previous implies that $\K$ is convex, and closed. Further, since $\phi\geq 0$ we have that $\K$ is nonempty as well since $0\in\K$. 

A few words are in order to establish the generality of the structure of \eqref{eq:KVI}. The class of spaces we have in mind are either of Lebesgue or Sobolev type. The possible choice of $G$ is contingent upon the choice of $V$; for example, if $V=H_0^1(
\Omega)$ then $G$ can be considered as the weak gradient $\nabla$, and if $H_0(\mathrm{div};\Omega)$, we can take $G=\mathrm{div}$. The function $\psi$ commonly refers to a $\ell^p$-norm in $\mathbb{R}^{\mathrm{m}}$, and the absolute value if $\mathrm{m}=1$, or to just the identity, i.e., $\psi(x)=x$. The regularity of $\phi$ is not an issue for well-posedness of variational problems over $\K$, but additional properties will be required for the obtention of stability results for perturbations of $\phi$.

{%
Note that  the expression \eqref{eq:KVI}, for given  functions 
$\psi$, $G$, and $\phi$, determines a fixed closed and convex set $\K$.  
However, for problems like quasi-variational inequalities (QVIs), the set $\K$ is actually a state-dependent quantity: This would lead to a problem like \ref{eq:VI} where $y\mapsto\K(y)$ is not constant. In terms of \eqref{eq:KVI} and this setting, the dependence of $\K(y)$ on $y$ is determined by assuming that  $\phi=\Phi(y)$; this is discussed later on.
}%
Based {on} the structure of \eqref{eq:KVI}, we have a general result under relatively weak conditions for \eqref{itm:2} in Definition \ref{definition:MoscoConvergence} to hold.

\begin{proposition}\label{iiMosco}
Suppose that  $\phi_n\to \phi$ in $L^1(\Omega)$. Then \eqref{itm:2} in Definition \ref{definition:MoscoConvergence} holds true for
\begin{equation}\label{eq:Kn}
    \mathbf{K}_n=\{w\in V: \psi(Gw) \leq \phi_n \}.
\end{equation}

\end{proposition}
\begin{proof}
For $w_n\in \K_n$, we have $\psi(Gw_n)\leq \phi_n$, and if $w_n\rightharpoonup w$  in $V$, it follows that $Gw_n\rightharpoonup Gw$  in $L^p(\Omega)^d$. By Mazur's lemma, there exists $z_n=\sum_{k=n}^{N(n)}\alpha{_k}(n) Gw_k$ where $\sum_{k=n}^{N(n)}\alpha{_k}(n)=1$ and $\alpha{_k}(n)\geq 0$ such that $z_n\to Gw$ in  $L^p(\Omega)^d$. Since $\psi\colon \mathbb{R}^d\to \mathbb{R}$ is convex,  
\begin{equation*}
	\psi(z_n)\leq \sum_{k=n}^{N(n)}\alpha{_k}(n) \psi(Gw_k)\leq \phi+\underbrace{\sum_{k=n}^{N(n)}\alpha{_k}(n) |\phi_k-\phi|}_{\epsilon_n}.
\end{equation*}
Since  $\phi_n\to \phi$ in $L^1(\Omega)$, then $\epsilon_n\to 0$ in $L^1(\Omega)${. Therefore,} we {obtain} $w\in  \K$ by taking the limit above (over some subsequence { converging in the pointwise almost everywhere sense}).
\end{proof}

On the other hand,  the 
existence of the subsequence in \eqref{itm:1} of Definition \ref{definition:MoscoConvergence} requires problem-specific constructions rendering it (in general) much harder to prove than  the condition in \eqref{itm:2} . {Perhaps the simplest situation in which} \eqref{itm:1} holds is  the obstacle case {with} $\phi_n\to \phi$ in $V$ and where $V\ni z\mapsto \min(0,z)\in V$ is continuous: Let $w\leq \phi$ be arbitrary and define $w_n:=\min(w,\phi_n)$ so that $w_n\leq \phi_n${.} {Since} $\phi_n\to \phi$ in $V$, it follows that $w_n\to w$ in $V$. { Consequently}  \eqref{itm:1}  holds true. The relaxation of `` $\phi_n\to \phi$ in $V$'' is a complex task that we tackle in some simple cases. We provide now some general constructions for \eqref{itm:1} and for Mosco convergence.

\begin{proposition}\label{iMosco} 
Let $V$ be either $W_0^{1,p}(\Omega)$ or $W^{1,p}(\Omega)$ 
	with $1\leq p< +\infty$, { and $\phi,\phi_n\in L^\infty(\Omega)$ for $n\in\mathbb{N}$}. Suppose that for some 
$\nu\in[0,\infty)$, 
 $\phi_n\geq \nu$ a.e. for all $n\in\mathbb{N}$ 
 	and  
	\begin{equation*}
		\phi_n\to \phi \quad \text{in}\quad L^\infty(\Omega).
	\end{equation*}
	 Then, if either {one of the two conditions}
	 \begin{align*}
	 	 & \mathrm{(i)}\quad \nu=0,\: G=\mathrm{id} \:\text{ and }\: \psi(x)=|x| \: \text{ or } \: \psi(x)=x \qquad \\[.1cm]
	 	& \mathrm{(ii)}\:\:\: \nu>0,\:  G=\nabla \:\text{ and }\:\psi(x)=\|x\|_{\ell^q} \: { with } \: 1\leq q\leq +\infty 
	 \end{align*}
	 hold true, we observe that 
	 \begin{equation*}
	 	\{w\in V: \psi(Gw) \leq \phi_n \} \Mosco\{w\in V: \psi(Gw) \leq \phi \} .
	 \end{equation*}
\end{proposition}
\begin{proof}
	Suppose that $\nu>0$. Note that since $\phi_n\to \phi$ in  $L^\infty(\Omega)$ by Proposition \ref{iiMosco}, \eqref{itm:2} in Definition \ref{definition:MoscoConvergence} holds true. In order to prove \eqref{itm:1}, let $w^*\in V$
 and  $\psi(Gw^*) \leq \phi$ be arbitrary. Define $w_n:=\beta_nw^*$ where
	\begin{equation*}
\beta_n:=\left(1+\frac{\|\phi_n-\phi\|_{L^\infty(\Omega)}}{{\nu}}\right)^{-1}.
\end{equation*}
It follows that $w_n\to w^*$ in $W^{1,p}_0(\Omega)$ and $\psi(Gw_n)\leq \phi_n$ {(cf. Hinterm{{\"u}}ller and Rautenberg~\cite{MR3023771})} which finishes the proof.

In the case $\nu=0$ for $\mathrm{(i)}$, consider
\begin{equation*}
	w_n=T_n(w^*):=\begin{cases} 
	 (\psi( w^*)-\|\phi-\phi_n\|_{L^\infty(\Omega)})^+\frac{w^*}{\psi(w^*)}, &\mbox{if } \psi(w^*)\neq 0; \\
0, & \mbox{if }  \psi(w^*)=0.\end{cases}
\end{equation*}
Note first that $w_n\in L^p(\Omega)$ and clearly $w_n\to w^*$ in $L^p(\Omega)$.
Further, note that $T_n(w^*)(x)=h_n(w^*(x))$ for some 
$h_n:\mathbb{R}\to\mathbb{R}$ and for each $n\in\mathbb{N}$. In this case, we have that $x\mapsto h_n(x)$ is globally Lipschitz, $h_n(0)=0$, then it follows that $T_n:V\to V$ is continuous {(see Marcus and Mizel~\cite{MR546508})}. 
Hence $\nabla w_n\in L^p(\Omega)$  and further
\begin{equation*}
	\nabla w_n= \nabla w^*\: T'_n(w^*) \qquad \text{with}\qquad T_n'(x):=\begin{cases} 
	 x/\psi(x), &\mbox{if } \psi(x)\geq \|\phi-\phi_n\|_{L^\infty(\Omega)}; \\
0, & \mbox{otherwise} .\end{cases}
\end{equation*} 
where we have used that $\psi(x)=x$ or $\psi(x)=|x|$. Hence,  $\nabla w_n\to \nabla w^*$ in $L^p(\Omega)$ { given that $T'_n(w)\to x/\psi(x)$ in $L^q(\Omega)$ for any $1\leq q<\infty$}, i.e., $w_n\to w^*$ in $V$. Finally, for $x\in \Omega$ we observe
\begin{equation*}
	\psi(w_n(x))=|(\psi( w(x))-\|\phi-\phi_n\|_{L^\infty(\Omega)})^+|\leq |(\psi( w(x))-\phi(x)+\phi_n(x))^+|\leq \phi_n(x),
\end{equation*}
which completes the proof.
\end{proof}

The case $\nu=0$ can also be handled in the gradient constraint case under additional assumptions on the regularity of the domain $\Omega$. Indeed, if 
$\Omega$ is bounded with $\partial\Omega$ of class $C^2$, then the result of 
the above theorem holds true in the $\mathrm{(i)}$ case, 
e.g. see Azevedo and Santos~\cite{Santos2004}.

\section{Further sufficient and necessary conditions for unilateral sets}\label{sec:sufnec}

In applications, it is common to encounter obstacle-type (or unilateral) constraints, i.e., 
\begin{equation}
	\K=\{w\in W_0^{1,p}(\Omega): w\leq \phi\}, \quad \text{ and } \quad 	\K_n=\{w\in W_0^{1,p}(\Omega): w\leq \phi_n\},
\end{equation}
for $1<p<+\infty${,} and $\Omega$ open and bounded. As we have shown before, it is simple to observe that if $\phi_n\to \phi$ in $W_0^{1,p}(\Omega)$,  then $\K_n\Mosco\K$. This, however, can be relaxed significantly and still preserve the Mosco convergence as we briefly discuss next.

The study of sufficient and of necessary conditions for convergence in the sense of Mosco has been an active area of research for several decades. In the case of unilateral sets and in $W_0^{1,p}(\Omega)$, a complete answer was given  by Dal Maso \cite{MR791845} where the condition involves properties on the $p-$capacities of the sets $\{x\in\Omega: \phi_n(x)<t\}$ and $\{x\in\Omega: \min(\phi(x),\phi_n(x))<t\}$. A similar capacitary approach was taken by Attouch and Picard \cite{MR683876,MR695419,MR653202} and sufficient conditions were obtained under stronger conditions than Dal Maso. 

As a sufficient condition for Mosco convergence, $\phi_n\to \phi$ in $W_0^{1,p}(\Omega)$ can be relaxed also substantially by means of the compactness result in Murat~\cite{MR633007} for Lipschitz domains that states: If $F_n\rightharpoonup F$ in $H^{-1}(\Omega)$ with $F_n\geq0 $ for all $n\in \mathbb{N}$, then $F_n\to F$ in $W^{-1,q}(\Omega)$ with $q<2$. Here, $F_n\geq0 $ {refers to} $\langle F_n,\sigma\rangle\geq 0$ for all $\sigma \in H_0^1(\Omega)$ with $\sigma\geq 0$. 
{Moreover}, the  Lipschitz regularity of $\partial\Omega$ can be dropped and the result {still remains intact; see} 
Br\'{e}zis~\cite{BrezisMur}. In our setting, 
{this result} leads to the following useful {assertion}; see Boccardo and Murat~\cite{MR1145746,MR652507}: If $\phi_n\rightharpoonup \phi$ in $W^{1,q}(\Omega)$ or $W_0^{1,q}(\Omega)$ 
for some $q>p$,  then $\K_n\Mosco\K$. In summary, Mosco convergence is maintained when switching from strong into weak convergence of the obstacles, provided that the gradients of the obstacles possess an $\epsilon>0$ extra amount of integral regularity.    

Analogous results to the one of Boccardo and Murat and of Dal Maso \cite{MR791845} were unknown for fractional spaces $W^{s,p}_0(\Omega)$ for $s\in (0,1)$ until recently {(see \cite{antil2020fractional})}. Applications for these kind of problems can be  seen in 
Antil and Rautenberg~\cite{antil2017fractional}. 
In the same vein, it is an open question whether it is possible to extend the above result of Boccardo and Murat~\cite{MR1145746,MR652507} 
to weighted Sobolev spaces $W_0^{1,p}(\Omega; w)$ for some $w$ 
in a Muckenhoupt class.

\section{The role of density in Mosco convergence}\label{sec:density}

This section entails a discussion on how density properties are related to Mosco convergence in regularization/discretization of optimization problems; we follow closely \cite{MR3710331}. 

In variational problems with constraints, one seeks the solution in a given convex, closed and nonempty 
feasible set $\K$ of a certain Banach space $(V,\|\cdot\|)$ not necessarily reflexive.  
To start the discussion in this section, 
let us consider the following abstract class of optimization problems:
\begin{equation}\label{prob:abstract}
  \min   F(u) \quad \text{over}\quad u\in \K,
\end{equation}
where assume that $F:V\to \mathbb{R}$ is continuous, coercive and sequentially weakly lower semicontinuous (not necessarily convex).     

Problem \eqref{prob:abstract} admits a solution provided $V$ is reflexive: Let $\{u_n\}$ be an infimizing sequence. Since $F$ is coercive and there is a feasible point, $\{u_n\}$ is bounded. Since $V$ is reflexive, $u_n\rightharpoonup u^*$ along a subsequence for some $u^*\in V$. Since $\K$ is convex and closed, it is weakly closed, and hence $u^*\in \K$. Finally, since $F$ is sequentially weakly lower semicontinuous, we have
\begin{equation*}
	F(u^*)\leq \liminf F(u_n)=\inf_{u\in\K}F(u),
\end{equation*}
i.e., $u^*$ is a minimizer, { and a subsequence of $\{u_n\}$ is not only a infimizing sequence but a minimizing one as well.}

The problem class \eqref{prob:abstract} is general enough to encompass numerous fields, such as 
 variational inequality problems of potential type, and optimal control of partial differential equations 
with constraints on the state and/or control among others. The study of \eqref{prob:abstract} and the design
of solution algorithms involve  
 concepts of perturbation or dualization methods comprising regularization, 
penalization or discretization
approaches (or a combination thereof). The stability properties of \eqref{prob:abstract} with respect to a large class of perturbations
is contingent upon the following density property:  For a {particular} dense subspace $Y$ of $V$, { it holds true that}
\begin{equation}\label{eqn:abstract_density}
	\overline{\{ u\in Y: u\in \K\}}^V=\K,
\end{equation}
or in short $\overline{\K\cap Y}^V=\K$. { Note that $\overline{\K\cap Y}^V$ refers to the closure in the $V$-norm of the set $\K\cap Y$.} 
In order to prove this, we start with the definition of $\Gamma$-convergence 
and its relation to Mosco convergence.

\begin{definition}
	Let $G_n:V\to \mathbb{R}\cup\{+\infty\}$ for $n\in\mathbb{N}$. We define  the $\Gamma$-upper and -lower limit at $u$ of $G_n$ as
	\begin{equation*}
		  \Gamma \text{-} \limsup_{n\to +\infty} G_n(u) := \sup_{U\in \mathcal{ N}(u)} \limsup_{n\to +\infty}  \inf_{w\in U} G_n(w), 
	\end{equation*}
	and
		\begin{equation*}
		  \Gamma\text{-}\liminf_{n\to +\infty} G_n(u) := \sup_{U\in \mathcal{ N}(u)} \liminf_{n\to +\infty}  \inf_{w\in U} G_n(w), 
	\end{equation*}
	respectively, where $\mathcal{ N}(u)$ denotes the set of all open neighborhoods in the norm of $V$. Analogously, we denote the weak versions of the above $\Gamma_w \text{-} \limsup G_n(u)$ and $\Gamma_w\text{-}\liminf G_n(u)$ where open neighborhoods are considered in the weak topology. Provided the limits  exists and are identical, we write 
		\begin{equation*}
		  \Gamma \text{-} \lim_{n\to +\infty} G_n(u) := \Gamma \text{-} \limsup_{n\to +\infty} G_n(u) = \Gamma\text{-}\liminf_{n\to +\infty} G_n(u), 
	\end{equation*}
	and say the quantity above is the (norm) $\Gamma$-limit of $G_n$ at $u$. Similarly, in the weak topology case, we define
	 \begin{equation*}
		  \Gamma_w \text{-} \lim_{n\to +\infty} G_n(u) := \Gamma_w \text{-} \limsup_{n\to +\infty} G_n(u) = \Gamma_w\text{-}\liminf_{n\to +\infty} G_n(u), 
	\end{equation*}
	provided the limits exist and are equal.
\end{definition}
	 
	 The connection of $\Gamma$-convergence and Mosco convergence is immediate. Consider the sequence $\{i_{\K_n}\}$ of indicator functions  $i_{\K}:V\to \mathbb{R}\cup\{+\infty\}$  for the sequence of convex closed and non-empty sets $\{\K_n\}$. Then, $\K_n\Mosco\K$ if and only if for each sequence $u_n\rightharpoonup u$ in $V$ with $u_n\in\K_n$, we have
\begin{equation*}
		  \liminf_{n\to +\infty} i_{\K_n}(u_n) \geq i_{\K}(u), 
\end{equation*}
	and for each $u\in \K$, there exists a sequence such $u_n\in \K_n$ such that $u_n\to u$ and 
\begin{equation*}
		   \limsup_{n\to +\infty} i_{\K_n}(u_n) \leq i_{\K}(u). 
\end{equation*}

We consider the above concepts applied to a general class of problems. For this matter, we define the sequence of perturbed problems 
\begin{equation}\label{prob:perturbed}
	\inf\quad F(u) + R_n(u), \quad \text{over } u\in V, 
\end{equation}
defined by given perturbations $R_n : V \to \mathbb{R}\cup \{+\infty\}$  
of the indicator function $i_{\K}:V\to \mathbb{R}\cup\{+\infty\}$  such that 
there exist functions $\underline{R}_n : X \to \mathbb{R}\cup \{+\infty\}$ and  
$\overline{R}_n : X \to \mathbb{R}\cup \{+\infty\}$ 
where
\begin{equation*}
	0 \le \underline R_n \le R_n \le \overline R_n \quad\forall n\in\mathbb{N}, 
\end{equation*}
and the additional properties hold

\begin{equation}\label{ass:lowerRn}
\begin{split}
& \underline{R}_n  \le \underline{R}_{n+1} \;\forall n\in\mathbb{N},\quad \lim_{n\to +\infty}\underline{R}_n(u) = i_{\K}(u) \;\forall u\in V ,\\  
& \underline{R}_n  \text{ sequentially weakly lower semicontinuous} \; \forall n\in\mathbb{N},
\end{split}	
\end{equation}
i.e., if $v_k\rightharpoonup v$ then $\underline{R}_n  (v)\leq \liminf_{k}\underline{R}_n (v_k)$, and 
\begin{equation}\label{ass:upperRn}
\overline{R}_n \ge \overline{R}_{n+1} , \; \forall n\in\mathbb{N}, \quad
\lim_{n\to +\infty}\overline{R}_n(u) = i_{\K\cap Y}(u) \quad\forall u\in V.   
\end{equation}
Mappings $R_n$ that  share the above features are usually called
\textit{quasi-monotone perturbations of the indicator 
function $i_{\K}$ with respect to the (dense) subspace $Y$}.  
We assume no additional assumptions for $R_n$ itself. 
In the stability analysis of \eqref{prob:perturbed}, 
the condition \eqref{eqn:abstract_density} appears immediately if using the theory of $\Gamma$-convergence 
(\cite{d93}): Under mild assumptions on $V$, the density property 
\eqref{eqn:abstract_density} ensures that $F+i_{\K}$ is the $\Gamma$-limit of $F+R_n$ in both, 
the weak and strong topology. { In this setting, the problem \eqref{prob:perturbed} admits a minimizer $u$, and 
each weak cluster point of any sequence of minimizers $\{u_n\}$ is a minimizer of \eqref{prob:abstract}; see Dal Maso \cite[Corollary 7.20]{d93}}.

We are now in position to establish the relation between $\Gamma$-convergence (and Mosco convergence) to the density property \eqref{eqn:abstract_density}

\begin{theorem}\label{prop:G-conv} 
Let $\{R_n\}$ be a sequence of quasi-monotone perturbations of $i_{\K}$ with respect to the 
dense subspace $Y$. Let the Banach space $V$ be reflexive or assume that $V^\ast$ is separable.  
If the density property \eqref{eqn:abstract_density} holds true, 
then $F+i_{\K}$ is the $\Gamma$-limit of $F+R_n$ in both, 
the weak and strong topology.  
\end{theorem}

\begin{proof}
Denote by $\mathrm{sc}^- G$ to the lower semicontinuous envelope of $G:V\to\mathbb{R}\cup\{+\infty\}$.
 From the relationship between $\Gamma$- and pointwise convergence \cite[Chapter 5]{d93},
with \eqref{ass:upperRn} and the continuity of $F$, we observe
\begin{align*}
	\Gamma_w \text{-}\limsup (F + R_n) & \le \Gamma \text{-}\limsup (F + R_n)\\ 
 & \le \Gamma \text{-} \limsup (F + \overline{R}_n) \\
 & = \mathrm{sc}^- (F + i_{\K\cap Y}) \\
 &= F + i_{\overline{\K\cap Y}},
\end{align*}
where we use \cite[Prop. 6.3, Prop. 6.7, Prop. 5.7, Prop. 3.7]{d93}. 

Analogously, \eqref{ass:lowerRn} together with \cite[Prop. 6.7, Prop. 5.4]{d93} leads to
\begin{equation}\label{eqn:scw}
	\Gamma_w\liminf (F + R_n)   \ge \Gamma_w\liminf (F + \underline{R}_n) = \lim_{n\to +\infty} \mathrm{sc}^-_w( F + \underline{R}_n) 
\end{equation}
where $\mathrm{sc}^-_w( F + \underline{R}_n)$  denotes the lower semicontinuous envelope of $F + \underline{R}_n$
in the weak topology of $V$. In addition, note that the coercivity and the sequential weak lower semicontinuity of 
$F+\underline{R}_n$ imply that the level sets $\{u\in V: F(u)+R_n(u) \le t\}$ are bounded and 
sequentially weakly closed. 
Since $V$ is reflexive or it has a separable dual $V^\ast$, then the sequential weak closure of 
bounded subsets coincides with the 
weak closure, see \cite[Prop. 8.7, Prop. 8.14]{d93}. Further, 
 $F+\underline{R}_n$ is weakly lower semicontinuous which determines  
\begin{equation*}
\Gamma_w\text{-}\liminf (F + R_n)  \ge  \lim_{n\to +\infty} (F + \underline{R}_n) = F + i_{\K}, 	
\end{equation*}
by \eqref{eqn:scw}.
Therefore, we observe that 
\begin{align*}
 F + i_{\K}  & \le \Gamma_w\text{-}\liminf (F+R_n)  \\ 
          & \le \Gamma_w\text{-}\limsup (F+R_n) \\
          &\le \Gamma \text{-}\limsup (F+R_n) \\
          & \le F+i_{\overline{\K\cap Y}},	
\end{align*}
such that 
$\Gamma\text{-}\lim(F+R_n) = \Gamma_w\text{-}\lim (F+R_n ) = F + i_{\K} $ if \eqref{eqn:abstract_density} holds true.
\end{proof}

{

In what follows, we provide a selection of approximation/regularization methods which fit 
into the general class of perturbations given by \eqref{prob:perturbed} and which are used very frequently in practice. 

\vspace{.2cm}

\begin{ex}[Tikhonov-Regularization]\label{ex:Tikhonov}
Let  $(Y,\|\cdot\|_Y)$ be a Banach space, and suppose that $Y$
is densely and continuously embedded into $V$. For a sequence of positive 
non-decreasing numbers  $\{\gamma_n\}$ 
with $\gamma_n\to +\infty$ and fixed $\alpha>0$, consider in \eqref{prob:perturbed} the Tikhonov regularization 
\begin{equation}
	R_n(u) = i_{\K}(u) + \tfrac{1}{2\gamma_n}\|u\|_Y^\alpha.
\end{equation}
We assume that  
$R_n(u)=+\infty$ if $u\notin Y$. Then, set 
\begin{equation*}
	\underline{R}_n := i_{\K}, \qquad \text{and} \qquad \overline{R}_n := R_n,
\end{equation*}
for all $n\in\mathbb{N}$, and \eqref{ass:lowerRn} and \eqref{ass:upperRn} are satisfied so that $R_n$ is in the 
context of \eqref{prob:perturbed}.
\end{ex}

\begin{ex}[Conformal discretization]\label{ex:Galerkin} Let $V$ be a separable Banach space.
Assume that \eqref{prob:abstract} is approximated by a Galerkin approach using nested and 
conformal finite-dimensional subspaces $V_n$, i.e., we have
$V_n\subset V$ and $V_n\subset V_{n+1}$ for all $n\in \mathbb{N}$ with the Galerkin approximation property:
\begin{equation*}
	\overline{\bigcup_{n\in\mathbb{N}} V_n}^V = V. 
\end{equation*}
Therefore, problem \eqref{prob:abstract} is replaced by \eqref{prob:perturbed} by the 
discretized counterpart defined by
 $R_n(u) = i_{\K\cap V_n}$. In this setting, define
 \begin{equation*}
 	\underline{R}_n := i_{\K}, \qquad \text{and}\qquad \overline{R}_n = R_n.
 \end{equation*}
It follows that \eqref{ass:lowerRn} is satisfied, and if 
 $Y := \bigcup_{n\in\mathbb{N}} V_n$, then
\eqref{ass:upperRn} is fulfilled as well.
\end{ex}

\vspace{.2cm}

\begin{ex}[Combined Moreau-Yosida-Tikhonov-Regularization]\label{ex:TikhonovMY}
Let $V$ be a Hilbert space and $(Y,\|\cdot\|_Y)$ a Banach space with $Y$ densely and continuously embedded into  $V$. For two sequences of positive 
non-decreasing numbers $\{\gamma_n\},\{\gamma_n'\}$ 
with $\gamma_n,\gamma_n'\to +\infty$ and fixed $\alpha>0$, consider the simultaneous Moreau-Yosida 
and Tikhonov regularization: 
\begin{equation}
R_n(u) = \tfrac{\gamma_n}{2}\inf_{v\in \K} \|u-v\|^2  + \tfrac{1}{2\gamma_n'}\|u\|_Y^\alpha,   
\end{equation}
with $\alpha>0$ fixed. We assume that 
$R_n(u)=+\infty$ if $u\notin Y$, and  define
\begin{equation*}
	\underline{R}_n(u) =  \tfrac{\gamma_n}{2}\inf_{v\in \K} \|u-v\|^2\qquad\text{and}\qquad\overline{R}_n(u) = i_{\K}(u) + \tfrac{1}{2\gamma_n}\|u\|_Y^\alpha.
\end{equation*} 
It is well-known from the theory of Moreau-Yosida regularizations that $\underline{R}_n$ 
satisfies \eqref{ass:lowerRn}, and
\eqref{ass:upperRn} it is also directly verified.
\end{ex}

\vspace{.2cm}

\begin{ex}[Conformal discretization and Moreau-Yosida regularization]\label{ex:GalerkinMY} 
Let $V$ be a separable Hilbert space and $\{\gamma_n\}$ a sequence of positive non-decreasing numbers with 
$\gamma_n\to+\infty$. The simultaneous regularization and discretization leads to 
\begin{equation}\label{def:GalerkinMY}
R_n(u) = \tfrac{\gamma_n}{2}\inf_{v\in \K} \|u-v\|^2 + i_{V_n}(u),	
\end{equation}
where the sequence of spaces $\{V_n\}$ is defined as in the previous examples. 
Defining
\begin{equation*}
	\underline{R}_n = \tfrac{\gamma_n}{2}\inf_{v\in \K} \|u-v\|^2  \qquad \text{and} \qquad \overline{R}_n = i_{\K\cap V_n},
\end{equation*}
\eqref{ass:lowerRn} and \eqref{ass:upperRn} are fulfilled with $Y = \bigcup_{n\in\mathbb{N}} V_n$ 
and the framework of \eqref{prob:perturbed} applies. 
\end{ex}

\vspace{.2cm}

From Theorem \ref{prop:G-conv}, the perturbations defined in the above examples  are stable with respect to \eqref{prob:abstract}
provided the density result \eqref{eqn:abstract_density} holds. 
Moreover, the density property \eqref{eqn:abstract_density} is also a necessary condition for
the stability of perturbation schemes in the following sense: Firstly, note that the $\Gamma$-limit 
of the approximation schemes defined in \Cref{ex:Tikhonov} and \Cref{ex:Galerkin} can be  
calculated using similar arguments as in the proof of \Cref{prop:G-conv}. 
Under the same conditions on $V$, namely that $V$ is reflexive or with separable dual, one infers that $F+i_{\overline{\K\cap Y}}$ is the weak and 
strong $\Gamma$-limit in both examples. 
Secondly, in the approaches of 
\Cref{ex:TikhonovMY} and \Cref{ex:GalerkinMY}, \Cref{prop:G-conv} guarantees that $F+i_{\K}$
is obtained as the weak-strong $\Gamma$-limit for \textit{any} coupling of 
regularization (parameter) pairs $[\gamma_n,\gamma_n']$ and $[V_n,\gamma_n]$, respectively. Further, in the combined Galerkin-Moreau-Yosida approach 
(\Cref{ex:GalerkinMY}), it is possible to 
prove the existence of a combination of $n$ and $\gamma_n$ to recover $F+i_{\K}$
in the $\Gamma$-limit without resorting to the density property \eqref{eqn:abstract_density}, 
see \cite[Prop. 2.46]{mr15}. However, the proof is non-constructive! Hence, it is not 
applicable for the design of solvers. Moreover, 
if \eqref{eqn:abstract_density} is violated, one may construct for any $w\in \K\setminus \overline{\K\cap Y}$
a sequence $\gamma_n$ such that no recovery sequence exists for the element $w$. The analogous statement is valid for the case of combined Moreau-Yosida-Tikhonov regularizations. Let us now rigorously establish the preceding statements. 

\begin{theorem}\label{prop:no_recovery}
 Let the assumptions of \textnormal{\Cref{ex:GalerkinMY}} be satisfied. Further suppose that 
 \begin{equation*}
 \overline{\K\cap Y}\subsetneq \K.	
 \end{equation*}
Then for each $w\in \K\setminus \overline{\K\cap Y}$ there exists a strictly increasing sequence of numbers
 $\{\gamma_n\}$ with $\gamma_n\to \infty$ such that 
 there exists no strong recovery sequence at $w$, i.e., 
 \begin{equation*}
 	F(y_n) + R_n(y_n) \nrightarrow F(w)
 \end{equation*}
for all $\{y_n\}$ in $V$ with $y_n\to w$, where $R_n$ is given by \eqref{def:GalerkinMY}.
\end{theorem}

\begin{proof}
 Let $w\in \K\setminus \overline{\K\cap Y}$ and $\rho > 0$ such that $\overline{B_\rho(w)} \cap \overline{\K\cap Y} = \emptyset$
 where $B_\rho(w):= \{y\in V: \|w-y\| < \rho \}$.

(a) \emph{We first prove the following result:}
\begin{equation}\label{def:g_n}
 \forall n\in\mathbb{N}, \: \exists \gamma_n > 0\:: \: \left[ \:y\in Y  \:\wedge\:  \mathrm{dist}(y,\K\cap \overline{B_{\rho}(w)})^2 < \tfrac{1}{\gamma_n} \quad \Longrightarrow \quad y\notin V_n \right]. 
\end{equation}
Assume the opposite, i.e.,
\begin{equation*}
 \exists n_0\in \mathbb{N }\quad :\quad \left[ \forall n\in \mathbb{N}, \exists w_n\in V_{n_0}, v_n \in (\K\cap \overline{B_{\rho}(w)})\quad : \quad \|w_n - v_n\|^2 \le \tfrac{1}{n} \right]. 
\end{equation*}
Since $v_n \in \overline{B_{\rho}(w)}\cap \K$ for all $n\in \mathbb{N}$ and $\overline{B_{\rho}(w)}\cap \K$ 
is convex, bounded and closed, there exists a subsequence $\{v_{n_k}\}$ of $\{v_n\}$ with  $v_{n_k}\rightharpoonup v$ and $v\in \overline{B_{\rho}(w)}\cap \K$. As $w_n-v_n \to 0$, one also obtains
$w_{n_k}\rightharpoonup v$ and thus $v\in V_{n_0}$. 
Hence, $v\in V_{n_0} \cap \K \cap \overline{B_{\rho}(w)} = \emptyset$, which is a contradiction. 

(b) \emph{Non-existence of a strong recovery sequence:} Choose $\gamma_n$ according to \eqref{def:g_n} and suppose there exists a recovery sequence $y_n$ to $w$, i.e., $y_n\to w$ and $F(y_n) + \frac{\gamma_n}{2} \mathrm{dist}(y_n,\K)^2 + i_{V_n}(y_n) \to F(w)$.
The continuity of $F$ implies that $y_n\in V_n$ and $\frac{\gamma_n}{2} \mathrm{dist}(y_n,\K)^2 \to 0$. 
Consequently, using $y_n\to w$ and $w\in \K$, there exists $n_1 \in \mathbb{N}$ such that  
\begin{equation*}
\mathrm{dist}(y_n,\K)^2 = \mathrm{dist}(y_n, \K \cap B_{\rho}(w))^2 \le \tfrac{1}{\gamma_n} 	
\end{equation*}
for all $n\ge n_1$. With the help of part $(a)$, we conclude that $y_n \notin V_n$ for all $n\ge n_1$  which is a contradiction.
\end{proof}

}

\subsection{Applications to Finite Element approximations}\label{sec:FE}

{ We now concentrate efforts in how the previously described ideas permeate through their finite dimensional approximation. In this section we assume that $\Omega\subset \R^\mathrm{d}$ is a Lipschitz polyhedral domain. We start with a small generalization of item \eqref{itm:1} in Definition \ref{definition:MoscoConvergence} for the finite dimensional case. In some textbooks on finite-dimensional approximations of variational inequalities, cf., e.g. Glowinski~\cite{g82},
Han and Reddy~\cite{hr13},
condition \eqref{itm:1} is commonly replaced by the following criterion: 
\vspace{.2cm}
\begin{enumerate}[\upshape(i)]
	\item \label{eqn:Mosco2} There exists a dense subset $\tilde{ \K} \subset \K$ and an operator 
$r_n:\tilde{\K} \to V$,    such that for all  $v\in\tilde{\K}$  it holds $r_n v \to v $ in $V$  and there exists  $n_0 \in \mathbb{N}$ such that $r_nv \in \K_n$  for all $n\ge n_0$. 
\end{enumerate}

It is easy to show that \eqref{eqn:Mosco2} implies \eqref{itm:1} in Definition \ref{definition:MoscoConvergence}. In fact,
let $v\in \K$ and denote by $ \pi_{ \K_n}v$ its (not necessarily uniquely determined) projection onto $\K_n$.
By density, for $\epsilon > 0$, there exists $v^\epsilon \in\tilde{ \K}$ such that $ ||v^\epsilon - v||\le \epsilon $.
Thus, we have 
 \begin{align*}
 	 || v - \pi_{ \K_n}v || & = \inf_{v^n\in \K_n} || v - v^n || \le || v - r_nv^\epsilon || \le \epsilon + || v^\epsilon - r_nv^\epsilon|| 
 \end{align*}
for sufficiently large $n$ such that $\lim_{n\to\infty }||v - \pi_{ \K_n}v || \le \epsilon $ where $\epsilon$ was arbitrary. 

Condition \eqref{eqn:Mosco2} is more convenient
 in the context of finite-dimensional approximations, where $r_n$ 
is given by interpolation operators that are only defined on a dense subset of $V$. Thus,
giving rise to sets $\tilde{ \K}$ of the type $\K\cap C_c^\infty(\Omega)$. 
in fact, 
this is precisely where the 
density results of the previous sections are required. 
For practical relevance, we consider the perturbation of
variational inequalities. 

In what follows, the sequence of approximating sets is assumed to be originating from a finite-dimensional
approximation $\K_n=\K_{h_n}$ of the set $\K$ in the framework of classical Finite Element methods: The parameter $n$ is associated with a sequence of mesh sizes $\{h_n\}$ 
converging to zero. Concerning the literature and in the context of approximation of variational inequalities,  
Falk~\cite{f74}'s a priori estimate for elliptic variational inequalities  
shows that it is sufficient to tailor the sets $\K_n$ with respect to 
the VI solution $u$: This gives rise to the class of adaptive Finite Elements methods. 
Rigorous convergence proofs for adaptive discretizations of 
variational inequalities are restricted to special cases and usually require strong assumptions. See for example, in the case of the obstacle problem with a piecewise
affine obstacle, the article Siebert and Veeser~\cite{sv07}. Furthermore, 
density results may still be useful in the analysis of
adaptive schemes utilizing interpolation operators, 
cf. Siebert~\cite{s11}.

Consider  a sequence of geometrically conformal affine simplicial meshes  
$\{\mathcal{T}_h\}_{h>0}$ of $\Omega$  of mesh size $h$, i.e.,  
\begin{equation*}
h:=\max_{T\in\mathcal{ T}_h} \diam\: (T),
\end{equation*}
where $\diam(T)=\max_{x,y\in T} |x-y|$ denotes the diameter of $T$. We call $\mathcal{ T}_h$, a triangulation of $\Omega$.
The Lebesgue measure of an element $T\in\mathcal{ T}_h$ is denoted by $\lambda(T)$.
We further assume that the sequence 
$\{ \mathcal{ T}_h\}$ is shape-regular, that is
\begin{equation}\label{def:shapereg}
\exists c>0 :\; \tfrac{\diam(T)}{ \rho_T} \le c \quad \forall h,\quad \forall T\in \mathcal{ T}_h,
\end{equation}
where $\rho_T$ is the diameter of the 
largest ball that is contained in $T$. Additionally, we write $x_T$ for the (barycentric) midpoint of an element $T$, and  
$\mathcal{M}_h=\{x_T:T\in\mathcal{T}_h\}$, $\mathcal{N}_h$ and $\mathcal{ E}_h$
for the set of element midpoints, triangulation nodes, and edges with respect to $\mathcal{ T}_h$, respectively. 
Abusing notation, we write $|\mathcal{M}_h|$ and $|\mathcal{N}_h|$ for the cardinality of the respective sets.
Let $\chi_T:\Omega\to\mathbb{R}$ be 
the characteristic function of $T$: 
\begin{equation*}
\chi_T(x) =  0 \quad \forall x\notin T, \quad \text{and} \quad \chi_T(x) =  1 \quad \forall x\in T.
\end{equation*}  
The standard $H^1(\Omega)$-conformal
Finite Element space of globally continuous piecewise affine functions associated to $\mathcal{ T}_h$ is given by
\begin{equation*}
H_h := \{ u \in C(\overline\Omega): u\vert_T \in \mathbb{P}_1 \;\: \forall\: T\in\mathcal{T}_h\}.
\end{equation*}
Here, $\mathbb{P}_1$ denotes the space of polynomials of degree less than or equal to one. Associated to $H_h$ and its standard nodal basis $\{\phi_x: x\in \mathcal{ N}_h\}$, 
we define the global interpolation operator
\begin{equation}\label{def:I_h}
I_h:C(\overline\Omega) \to H_h, \quad I_hu := \sum_{x\in\m N_h} u(x)\phi_x. 
\end{equation}
Note that $I_h$ is only defined on a dense subspace of $H^1(\Omega)$.

We define the Hilbert space $H(\diver, \Omega):=\{v\in L^2(\Omega)^N: \diver v \in L^2(\Omega)\}$ endowed with the inner product
\begin{equation*}
(v,w)_{H(\diver)}:=(v,w)_{L^2(\Omega)^N}+(\diver v,\diver w)_{L^2(\Omega)}.
\end{equation*}
The closure of $C_c^\infty(\Omega)^N$ with respect to the $H(\diver, \Omega)$-norm is denoted by  $H_0(\diver,\Omega)$ and in the case $\Omega$ has a Lipschitz boundary it is equivalent to
\begin{equation}\label{eq:}
H_0(\diver,\Omega)=\{v\in H(\diver,\Omega):\quad \gamma v:= v\cdot{\nu}|_{\partial \Omega}=0\},
\end{equation}
where ${\nu}$ denotes the outer normal vector. The operator $\gamma$ can be proven to be continuous from $H(\diver,\Omega)$ to $H^{1/2}(\partial \Omega)$. For the discretization of variational
problems in $\hdiv$, it is usual to consider the $\hdiv$-conforming space of Raviart-Thomas Finite Elements of lowest order:
\begin{equation}\label{def:RT}
RT_h = \{ w\in L^2(\Omega)^\mathrm{d}: w\vert_T \in \mathbb{RT} \:\:\: \forall T\in\m T_h,\:\:\: [w\cdot\nu]\vert_E = 0 \:\:\:\forall E\in \m E_h\cap\Omega \},
\end{equation}
where $\mathbb{RT} = \{ w\in \mathbb{P}_1^\mathrm{d} : \exists a\in\mathbb{R}^\mathrm{d},b\in\mathbb{R}, \text{ for which } w(x) = a + bx\}$ and $\nu$ denotes the unit 
outer normal to $T$. The incorporation of zero boundary conditions in the normal direction requires the use of the $H_0(\mathrm{div},\Omega)$-conforming subspace 
\begin{equation*}
RT_{0,h} := RT_h \cap H_0(\mathrm{div},\Omega).
\end{equation*} 
Suitable edge-based basis functions $\{\phi_E: E\in\m E_h\}$ 
can be found in the literature, cf., 
for instance, Bahriawati and Carstensen~\cite{cb05}.
Finally, the global Raviart-Thomas interpolation operator is given by 
\begin{equation}\label{def:IRT}
I_h^{ RT}:  W^{1,1}(\Omega)^\mathrm{d}\to RT_h, \quad I^{ RT}_h w := \sum_{E\in\m E_h}\left( \int_E w\cdot \nu \dif \mathcal{H}^{\mathrm{d}-1} \right) \; \phi_E. 
\end{equation} 

We are now in shape to present the pertinent Mosco convergence results associated to finite element discretizations.
\begin{thm}\label{thm:Mosco1}
 Suppose that  $\alpha\in C(\overline\Omega)$ satisfies $\inf_{x\in \Omega}\alpha(x)>0$, and that $N\in\mathbb{N}$ is given.  
 Then the sets 
 \begin{align*}
    \K^1_h&: = \{w\in (H_h)^N: |w(x_T)|_{\ell^p} \le \alpha(x_T) \text{ for all } T\in\m T_h \},\\
     \K^2_h &:= \{ w\in (H_h)^N : |w(x)|_{\ell^p}\le \alpha(x) \text{ for all }  x\in \mathcal N_h \},
 \end{align*}
 for $1\leq p \leq +\infty$, Mosco-converge for $h\to 0$ to the set 
 \begin{equation*}
 	\K=\{w\in H^1(\Omega)^N: |w|_{\ell^p}\leq \alpha\}
 \end{equation*}
in $H^1(\Omega)$.
\end{thm}

\begin{proof}
We concentrate on $\K^1_h$ as the proof for $\K^2_h$ follows analogously, and we separate the proof into two steps.

\emph{Step 1: We prove first that \eqref{itm:2} in Definition \eqref{definition:MoscoConvergence} holds true.} That is, suppose  $w_h\in \K^1_h$ and $w_h\rightharpoonup w$ in $H^1(\Omega)$ along a subsequence, then we prove that $w\in \K$.  
 It suffices to show that $i_{\K}(w)=0$.
  Furthermore, it holds $i_{\K} = j^\ast$ where $j^\ast$ denotes the 
 Fenchel-Legendre conjugate 
 \begin{equation*}
 	  j^\ast(v^\ast) := \sup_{v \in L^2(\Omega)}\{(v^\ast,v) - j(v)\},
 \end{equation*}
 of the map $j:L^2(\Omega) \to \R,$ defined as 
 \begin{equation*}
 	j(v):= \int_\Omega \alpha|v|_{\ell^q} \dif x,
 \end{equation*}
with $1/p+1/q=1$ , and where we use the duality relation between ${\ell^p}$ and ${\ell^q}$ norms, i.e.,
\begin{equation*}
	 |v^\ast|_{\ell^p} = \sup_{v\in\R\setminus\{0\}} v^\ast \cdot v / |v|_{\ell^q}.
\end{equation*}

 From the definition of $j^\ast$, we obtain that
 $i_{\K}(w)=0$ is equivalent to
 \begin{equation}\label{est:proof1}
 	(w,v) \le \int_\Omega \alpha |v|_{\ell^q} \quad\fa v\in L^2(\Omega). 
 \end{equation}
 Via density, it is enough
 to prove this result for all $v \in C_c(\Omega)$. 
Define
  \begin{equation}\label{def:pwcInterpolate}
  	\alpha_h:= \sum_{T\in\m T_h}\alpha(x_T)\chi_T, \quad \text{and} \quad v_h:= \sum_{T\in\m T_h}v(x_T)\chi_T 
  \end{equation}
the piecewise constant interpolants of $\alpha$ and $v$, respectively. Since $\alpha$ and $v$ are uniformly continuous, then
 $\alpha_h\to \alpha$ and $v_h\to v$ in $L^\infty(\Omega)$. {By the weak convergence of $w_h$, and the strong convergence of
 $\alpha_h$ and $v_h$, we have
  \begin{align}
 \int_\Omega w_h \cdot v_h \dif x &\xrightarrow{ h\downarrow 0}  \int_\Omega w\cdot v \dif x ,\label{eq:con1}\\ 
  \int_\Omega \alpha_h |v_h|_{\ell^q} \dif x  &\xrightarrow{ h\downarrow 0} \int_\Omega \alpha |v|_{\ell^q} \dif x .\label{eq:con2}
 \end{align}}
 { Further, by the midpoint quadrature rule, and that $|w_h(x_T)|_{\ell^p} \le \alpha(x_T)$, we observe that 
 \begin{align*}
 \int_\Omega w_h \cdot v_h \dif x & = \sum_{T\in \m T_h} \int_T w_h \cdot v_h\dif x \nonumber\\ 
 & = \sum_{T\in \m T_h} \lambda(T)\, w_h(x_T) \cdot v_h\vert_T \dif x \label{eqn:midpointformula}\\
  & \leq \sum_{T\in \m T_h} \lambda(T)\, |w_h(x_T)|_{\ell^p} |v_h\vert_T|_{\ell^q} \dif x \\ 
 & \le \sum_{T\in \m T_h} \lambda(T)\; \alpha(x_T) \; |v_h\vert_T |_{\ell^q} \dif x \nonumber \\
 & = \int_\Omega \alpha_h |v_h|_{\ell^q} \dif x  \nonumber,	
 \end{align*}}
which {by \eqref{eq:con1} and \eqref{eq:con2}} proves \eqref{est:proof1}. 

\emph{Step 2: We prove  that \eqref{itm:1} in Definition \eqref{definition:MoscoConvergence} holds true.} Note that  the assumptions on $\alpha$ imply that
\begin{equation*}
	\overline{\K\cap C_c^\infty(\Omega)^N}^{H^1(\Omega)^N}=\K,
\end{equation*}
{that is, the set $\K\cap C_c^\infty(\Omega)^N$ is dense, with respect to the $H^1(\Omega)^N$-norm, in $\K$;} see Hinterm{\"u}ller and Rautenberg~\cite{MR3306389}. { This further  implies} that the set 
\begin{equation}
\label{def:Ktilde}
 \tilde{\K} := \{ \phi\in C^\infty(\overline\Omega)^N: |\phi(x)|_{\ell^p}<\alpha(x) \text{ for all } x\in \Omega \},
\end{equation}
is also dense in $\K$ w.r.t. the $H^1(\Omega)^N$-norm. 
 For the global interpolation operator $I_h$ defined in \eqref{def:I_h} we have the classical estimate, 
 \begin{equation}
 	\label{est:interpolation}
|| u - I_h u ||_{ L^\infty(\Omega)} \le c h^2 || u ||_{ W^{2,\infty}(\Omega)} \quad\fa u\in W^{2,\infty}(\Omega).
 \end{equation}
Here, $c$ denotes a constant independent of $h$ on account of the shape-regularity of the triangulation 
\eqref{def:shapereg}; see Ern and Guermond~\cite{eg04}.
 
We set $r_h:\tilde{\K} \to (H_h)^N$ to be defined by $r_h w = \{I_h w_i\}_{i=1}^N$ and it follows
that $r_hw\to w$ as $h\to 0$ in $H^1(\Omega)^N$ for all $w\in \tilde{ \K}$. Hence,
\begin{equation}
	\label{est:proof2}
||\, |w - r_h w|_{\ell^p}\, ||_{ L^\infty(\Omega)} \le \tilde{c}h^2 ||w||_{ W^{2,\infty}(\Omega)^N},
\end{equation}
for some $\tilde{c}>0$, which implies
\begin{equation}
	\label{est:proof3}
|r_h w(x)|_{\ell^p} \le |w(x)|_{\ell^p} + ch^2 ||w||_{ W^{2,\infty}(\Omega)^N} \quad \fa x\in\Omega. 
\end{equation}
Thus for $w\in \tilde{\K}$, there exists $h_0=h_0(w)$ such that $r_h w \in \K^1_h$ for all  $h\le h_0$ which implies \eqref{eqn:Mosco2}.
\end{proof}

The role of density properties can also be seen in the following result involving other kinds of constraints. The proof carries over mutandis mutatis from the previous proof. 

\begin{thm}\label{thm:Mosco2}
 Let  $1\leq p \leq +\infty$, and  assume that $\alpha\in C(\overline\Omega)$ satisfies $\inf_{x\in \Omega}\alpha(x)>0$.  
 Then the set 
 \begin{align*}
    &\K^i_h: = \{ w \in H_h: |\nabla w\vert_T|_{\ell^p} \le \alpha(x_T) \;\text{for all } T\in\m T_h\},
 \end{align*}
 Mosco-converges for $h\to 0$ to the set
 \begin{align*}
 	&\K^i=\{w\in H^1(\Omega): |\nabla w|_{\ell^p}\leq \alpha\},
 \end{align*}
in $H^1(\Omega)$. Further, the sets 
 \begin{align*}
     & \K^{ii}_h :=\{ w\in RT_{0,h}: |w(x_T)|_{\ell^p} \le \alpha(x_T) \;\text{for all } T\in \mathcal{T}_h\}, \\
     & \K^{iii}_h :=\{ w\in RT_{0,h}: |\mathrm{div}\: w\vert_T| \le \alpha(x_T) \;\text{for all }  T\in \m T_h\},
 \end{align*}
Mosco-converge for $h\to 0$ to the sets 
 \begin{align*}
 	&\K^{ii}=\{w\in H_0(\mathrm{div},\Omega): | w|_{\ell^p}\leq \alpha\},\\
 	 &	\K^{iii}=\{w\in H_0(\mathrm{div},\Omega): | \mathrm{div}\: w|\leq \alpha\},
 \end{align*}
in $H_0(\mathrm{div},\Omega)$.\qed
\end{thm}

}

\section{Quasi-variational inequalities}\label{sec:QVIs}

The structure \eqref{eq:KVI} of $\K$ is adapted to VIs, where the convex set $\K$ is part of the fixed data.  However, 
to treat Quasi-Variational Inequalities (QVIs) we need to consider the convex set as unknown a-priori.  Therefore, instead of a convex $\K$ we have a map $\K:V\to 2^V$ written as 
$v\mapsto \mathbf{K}(v)$ with the following difference with respect to  
\eqref{eq:KVI}, the function $\phi$ is contingent upon the state $y$ as well.  Indeed, in the unilateral case, there is an operator $\Phi$ such that  
$\Phi(v):\Omega \to\mathbb{R}$ is measurable function for each $v\in V$, and 
\begin{equation}\label{eq:KQVI}
    \mathbf{K}(v):=\{w\in V: w \leq \Phi(v) \}.
\end{equation}
Thus, if $\K(\cdot)$ is as above and $f\in V'$ is given, then
\begin{equation}\label{eq:QVI}\tag{$\mathrm{\mathrm{P}_{QVI}}$}
\text{{Find} } y\in \mathbf{K}(y): \langle Ay-f,v-y\rangle \geq 0, \quad \forall v\in \mathbf{K}(y),
\end{equation}
is referred to as a QVI. This kind of problems arose initially from the work of Bensoussan and Lions \cite{Bensoussan1974,Lions1973} (see also the monographs \cite{Ben1982,Bensoussan1984}) on \emph{impulse control problems}, and later found application modeling a wide variety of non-convex and non-smooth phenomena in applied sciences. Specifically, areas including superconductivity (Kunze and Rodrigues~\cite{Rodrigues2000},
Rodrigues and Santos~\cite{MR1765540,MR2947539},
Barrett and Prigozhin~\cite{MR2652615,MR3335194},
Prigozhin~\cite{Prigozhin},
Hinterm{\"u}ller and Rautenberg~\cite{MR3648950,MR3119319,MR3023771,hintermuller2019dissipative}), continuum mechanics (Friedman~\cite{Friedman1982}), growth of sandpiles
(Barrett and Prigozhin~\cite{MR3082292,MR3231973,MR3335194},
Prigozhin~\cite{Prigozhin1986,Prigozhin1994,Prigozhin1996}), and the determination of rivers/lakes networks 
(Barrett and Prigozhin~\cite{MR3231973},
Prigozhin~\cite{Prigozhin1994,Prigozhin1996}), among others. For a complete and classical account on QVIs, we refer the reader to the text of Baiocchi and Capelo \cite{BaiC1984}.

\subsection{Impulse Control Problems and QVIs}\label{sec:impulse}

Because this is an application in stochastic control problems, the proper description of the impulse control involves diffusion processes as the state of the system to be controlled, and a complete setting can be found, for instance, in most of the quoted references below.

In the simplest case, impulse control (or control by interventions) refers to a sequential choices of 
parameters that modify the free evolution of the system, e.g., beginning at time $\theta_0=0$ and 
a state $x_0\in\R^d$, the state $\{x(t):t\ge\theta_0\}$ system is allowed to evolve with a running cost given by 
$f(x(t))e^{-\alpha t}$  (assuming $a_0(x)=\alpha$, constant) until a time 
$\theta_1$, where the controller 
intervenes and changes  the state and/or evolution of the system, e.g., if the current state 
is $x(\theta_1)$ then immediately, the state is moved to the state 
$x(\theta_1)+\xi_1$, $\xi_1\ge0$ and the evolution continues with a similar law.
For instance, in finance, the state $x(t)$ may represent the inventory at time $t$ and $\xi_1$ 
the order placed at time $t=\theta_1$.   Iterating this, a control policy $\{(\theta_i,\xi_i):i\ge1\}$ 
is obtained, and the control problem could be properly defined. 

In the context described above and subsequently, the dynamic programming is applied to obtain the so-called Hamilton-Jacobi-Bellman equation, which takes the form of a QVI. In particular, if $y$ is the optimal cost, then at any given time the controller has to decide whether to 
continue the (free) evolution, i.e., following the equation $Ay=f$, or to make an impulse (intervention), 
which has a cost (and changes $y$ into $My$).  This can be accounted as
\[
  \textrm{(a) }\;Ay\le f,\qquad \textrm{(b) }\;y\le My,\qquad \textrm{(c) }\;(Ay-f)(y-My)=0,
\]
where  $A$ is a second order elliptic operator with Lipschitz continuous and bounded 
coefficients in a smooth domain $\Omega$ of $\mathbb{R}^d$, i.e.,
\[
  Ay=-\sum_{i,j=1}^d a_{ij}(x)\partial_{ij}y(x) +\sum_{i=1}^da_i(x)\partial_i y(x)+a_0(x)y(x),
\]
where $(a_{ij})$ and $(a_i)$ are related with the diffusion and drift terms and the operator $M$ takes the form
\[
  My(x)=\inf\big\{y(x+\xi)+k(\xi):\xi\ge0\big\},
\]
for a suitable function $k(\xi)\ge k_0>0$ representing the cost-per-impulse. 
Usually, there may be more that one solution of these inequalities, even the \emph{complementary} 
condition (c) is not enough to ensure uniqueness in a general setting.  
Moreover, adding those other conditions, a minimum (minimal or maximal, depending on the setting) solution satisfying (a) and (b) is found. 
In variational form, this is equivalent to \eqref{eq:QVI}, and the perturbation of extremal solutions thereof is treated on \S \ref{sec:Order}.

The expression of the operator $M$ can be modified to deal with more complex settings, e.g., if a fixed time delay 
$\tau$ is imposed (i.e., $\theta_{i+1}\ge\theta_i+\tau$) then
\[
  My(x)=\inf\big\{\mathbb{E}_x\{y(x(\tau)+\xi)+k(\xi)\}:\xi\ge0\big\},
\]
where $\mathbb{E}_x$ is the expectation given $x(0)=x$.  In general, the region $\{\xi\ge 0\}$ may be replaced by a subset $\Gamma(x)\subset\Omega$ depending 
on the given $x$.  Moreover, the whole state space $\Omega$ can be divided into three regions, 
where (1) impulses are not allowed, (2) impulses are allowed, and (3) impulses are required; e.g., (3) is a piece $\Gamma_0$ of the boundary of $\Omega$,  (2) is the interior of $\Omega$ and (1) is 
the complement of $\Gamma_0$ (or empty).  
In this case, the expression of $M$ changes considerable, but some of the essential properties 
(e.g., like its monotone character) are retained.   
This last example is included in the so-called \emph{hybrid models}, where discrete and continuous 
type variables are used, e.g., examples of this situation can be found in 
Bensoussan and Menaldi~\cite{BM1997,BM2000}, as well as particular cases in more recent papers
Menaldi and Robin~\cite{MR2017, MR2018}, among others. For degenerate problems the reader may check \cite{Me1987}, and applications to Navier-Stokes are considered by Menaldi and Sritharan~\cite{MeSr2003}.

A vector form goes under the name of \emph{switching control}, the coefficients of the diffusion  depend on a parameter $i=1,2,\ldots,\bar{n}$, i.e., the operator $A$ becomes $A_i$ and a simple 
expression for $M$ takes the form
\[
  Mv(x,i)=\inf\big\{v(x,j)+k(i,j)\}:j\neq i\big\},\quad k(i,j)\ge 0,
\]
which can be combined with previous forms of $M$.  There is a vast literature on these problems,
as recent books, the reader may consult 
Arapostathis et al.~\cite{ArBM2012}, Yin and Zhu~\cite{YiZ10}, among others; and for instance, a relative complex situations is discussed in Menaldi and Robin~\cite{MR2019}.

\subsection{Elementary Existence Theory}\label{sec:existence}

For the study of existence of solutions, we define the map 
\begin{equation}\label{eq:T}
T(v):=S(f,\K(v)),
\end{equation}
where $S$ is the solution map associated to the variational inequality \eqref{eq:VI}, relative to $f$ and $\K(v)$.
Thus, solutions to  \eqref{eq:QVI} are equivalently defined as fixed points of the map $T$, i.e., $v$ solves \eqref{eq:QVI} iff 
\begin{equation*}
	T(v)=v.
\end{equation*}

A direct approach to determine existence of fixed points is the following. The coercivity of the operator $A$ implies that $T(V)\subset B_R(0; V)$ for some $R>0$. Hence, any sequence $\{v_n\}$ in $B_R(0; V)$ contains a subsequence such that $v_n\rightharpoonup v $ and  $T(v_n)\rightharpoonup z $ in $V$ for some $v$ and $z$. Hence, provided that $\K(v_n)\Mosco\K(v)$ then $T(v_n)\to T(v) $ in $V$, i.e., the map $T$ is compact and a fixed point exists due to the theorem of Schauder. In summary, a sufficient condition for the existence of solutions to \eqref{eq:QVI} is that $v_n\rightharpoonup v $ in $V$ implies that $\K(v_n)\Mosco\K(v)$.

While the above is suitable to understand the problem of existence, it is not enough to understand the behavior of the set of all solutions $\mathbf{Q}(f)$ to \eqref{eq:QVI} with respect to perturbations of $f$. For this, we consider an ordering approach.

\subsection{Exploiting order and cone structure} \label{sec:Order}

We consider an approach based on order that was pioneered by Tartar; see \cite{tartar1974inequations}, \cite[Chapter 15, \S 15.2]{Aubin1979}, and we follow closely a simplified version of \cite{alphonse2019stability,Alphonse_2019}. In particular, we focus on existence and stability properties of the solution set.

Let  $(V,H,V')$ be a Gelfand triple of Hilbert spaces, that is, we have $V\hookrightarrow H \hookrightarrow V'$, where the embedding  $V\hookrightarrow H$ is dense and continuous, and $H$ is identified with its topological dual $H'$ so that the embedding $H \hookrightarrow V'$ is also dense and continuous. Within this section, $(\cdot,\cdot)$ denotes the inner product in $H$. 

We assume that $H^+\subset H$ is a convex cone satisfying
\begin{equation*}
H^+=\{v\in H: (v,y)\geq 0 \text{ for all $y\in H^+$}\}.
\end{equation*}
Note that $H^+$ defines the cone of non-negative elements inducing the vector ordering:
\begin{equation*}
x\leq y \quad\text{ if and only if }\quad y-x\in H^+. 
\end{equation*}
Given $x\in H$, let $x^+$ denote the orthogonal projection of $x$ onto $H^+$, and define $x^-:=x^+-x$. Clearly, one has the decomposition $x=x^+-x^-\in H^+-H^+$ for every $x\in H$, and $(x^+,x^-)=0$. Further, the infimum and supremum of two elements $x,y\in H$ are defined as $\sup (x,y):=x+(y-x)^+$ and $\inf (x,y):=x-(x-y)^+$, respectively. The supremum of an arbitrary {completely ordered subset $R$} of $H$ that is bounded (in the order) above is also properly defined:  $R$ can be written as $\{x_i\}_{i\in J}$, where $J$ is completely ordered, { and it follows that} $\{x_i\}_{i\in J}$ is a generalized Cauchy sequence in $H$ 
(e.g., see Aubin~\cite[Chapter 15, \S15.2, Proposition 1]{Aubin1979}); 
its limit is the upper bound of the original set.  {Additionally, we have that} that norm convergence preserves order, i.e., if { $z_n\to z$ and $y_n\to y$ in $H$, then $z_n\leq y_n$ ($y_n-z_n\in H^+$) implies $z\leq y$, since $H^+$ is closed}. 

We further assume that 
\begin{equation*}
y\in V \Rightarrow y^+\in V \qquad \text{and} \qquad \exists \mu>0:  \|y^+\|_V\leq \mu \|y\|_V, \:\:\forall y\in V.
\end{equation*}
{Then the order} in $H$ induces one in $V'$, {as well}. {In fact, for}  $f,g\in V'$, we write $f\leq g$ if $\langle f,\phi\rangle\leq \langle g,\phi\rangle$ for all ${\phi}\in V^+:=V\cap H^+$.

Finally, $V$ and $H$  are assumed to be spaces of maps $h:\Omega\to \mathbb{R}$ over some open set $\Omega\subset\mathbb{R}^N$ with the following dense and continuous embedding: $L^\infty(\Omega)\hookrightarrow H$ such that $L^\infty(\Omega)\hookrightarrow V'$, as well. Additionally, we assume that $H\hookrightarrow L^1(\Omega)$.

A common example of Gelfand triple $(V,H,V')$ and cone $H^+$ that satisfies all conditions is given by  $(V,H,V')=(H_0^{1}(\Omega),L^2(\Omega), H^{-1}(\Omega))$ with $H^+=L^2(\Omega)^+$, the set of almost everywhere (a.e.) non-negative functions, and $v\leq w$ in the a.e. sense.

\subsubsection{Minimal and Maximal Solutions} We start this section with the definition of an increasing map, and existence of fixed points thereof under rather weak conditions. Subsequently, we provide conditions for the map $T$ to be increasing.
{
\begin{definition}
A map  $R\colon H\to H$ is said to be increasing if for $y,z\in V$ we have that
\begin{equation*}
y\leq z \quad \text{implies} \quad R(y)\leq R(z).
\end{equation*}
\end{definition}}

A general result {concerning} existence of fixed points for increasing maps is  available as we see next (its proof can be found on \cite{Aubin1979}). This provides a fundamental tool to prove existence of solutions to problem \eqref{eq:QVI} under very weak assumptions.

\begin{thm}[\textsc{Birkhoff-Tartar}]\label{thm:BirkhoffTartar} 
Suppose $R:H\rightarrow H$ is an increasing map and let $\underline{y}$ be a sub-solution and $\overline{y}$ be a  super-solution of the map $R$, that is:
\begin{equation*}
\underline{y}\leq R(\underline{y}) \quad \text{ and } \quad R(\overline{y})\leq \overline{y}.
\end{equation*}
If  $\underline{y}\leq  \overline{y}$, then the set of fixed points of the map $R$ in the interval $[\underline{y}, \overline{y}]$ is non-empty and has a smallest and a largest element.\qed
\end{thm}

The above theorem mainly {states} that if a map is increasing, has a subsolution $y_1$ and a supersolution $y_2$, then it has a fixed point between (with respect to the order induced in $H$) $y_1$ and $y_2$. {Moreover}, there are minimal and maximal fixed points in $[y_1, y_2]$.

For the map $T\colon H\to H$ defined as $T(v)=S(f,\K(v))$ to be increasing, some assumptions are required on the structure of $\K$ and on the operator $A$. {For this purpose}, in addition to $A:V\to V'$ {satisfying \eqref{eq:AssA} ( i.e., $A$ is linear, continuous, and strongly monotone), we assume it is strictly T-monotone, i.e.,
\begin{equation}\label{A3}
\langle A y^ -, y^+ \rangle\leq 0, \qquad \forall y\in V.
\end{equation}}
Further, we assume that 
\begin{equation}
	\Phi\colon H\to H\cap [\nu,+\infty) \quad \text{ is increasing},
\end{equation}
for some $\nu>0$, and {that} $0\leq f\leq f_{\mathrm{max}}$ for some $ f_{\mathrm{max}}\in V'$.  Then, it follows that
\begin{equation*}
\underline{y}=A^{-1}0=0 \qquad \text{and} \qquad \overline{y}=A^{-1}f_{\mathrm{max}}
\end{equation*}
are sub- and supersolutions, respectively, of $T$, and all assumptions of the previous theorem are satisfied: In fact, we have that 
\begin{equation*}
	(f,v)\mapsto S(f,\K(v)) \qquad \text{ is increasing},
\end{equation*}
see 
Rodrigues~\cite[Section 4:5, Theorem 5.1]{Rodrigues1987}. Hence, defining 
$\mathbf{A}_{\mathrm{ad}}=\{g\in V': 0\leq g\leq f_{\mathrm{max}}\}$, we have the operators 
\begin{equation*}
\mathsf{m}\colon \mathbf{A}_{\mathrm{ad}}\to V \qquad \text{and} \qquad \mathsf{M}\colon \mathbf{A}_{\mathrm{ad}}\to V
\end{equation*}
that take elements of $\mathbf{A}_{\mathrm{ad}}$ to minimal and maximal solutions to \eqref{eq:QVI} {in} the interval $[\underline{y},\overline{y}]=[0,A^{-1}f_{\mathrm{max}}].$
 
\subsubsection{A class of QVIs } 

Consider the following class of compliant obstacle problems where the obstacle is given implicitly by solving a PDE, thus coupling a VI and a PDE. It consists in finding $(y,\Phi,z)\in V\times H\times W$ such that 
\begin{align}
 &y \leq \Phi, \quad \langle A(y)-f, y-v \rangle \leq 0, 
   \qquad \forall v \in V: v \leq \Phi,\label{eq:QVIforu}\\
 &\langle Bz+G(\Phi,y)-g, w\rangle=0 
   \qquad \forall w\in W,\label{eq:pdeForT}\\
 &\Phi=Lz, \qquad\text{in } H.\label{eq:mould}
\end{align}
Here, $V\hookrightarrow W\hookrightarrow H \hookrightarrow W'\hookrightarrow V'$, $f,g \in H^+$, $G:H\times H\to H$ is continuous and bounded, i.e., for some $M_G>0$, $\|G(\Phi,y)\|_H\leq M_G (\|\Phi\|_H+\|y\|_H)$, for all $(\Phi,y)\in H\times H$. Further, $L\colon W \to H$ is an increasing {affine} linear continuous map { with $L(0)\geq\nu>0$}. Additionally,  $B\in \mathcal{L} (W,W')$ is coercive and satisfies $\langle Bz^-,z^+\rangle= 0$ for all $z\in W$ { (i.e., $B$ is T-monotone).}

Under mild conditions, the above problem can be cast into the form of \eqref{eq:QVI} as follows. Let $v\in H$, and consider the problem of finding $z\in W$ such that
\begin{align}
  &\langle Bz+G(\phi,v)-g, w\rangle=0, \quad\forall w\in W,\label{eq:mh1}\\
  &\phi = Lz, \quad\text{in } H.\label{eq:mh2}
\end{align}
Assuming that for each $v\in H$, $z\mapsto G(Lz,v)$ is monotone, one can show the existence of a unique solution $z(v)\in W$ of \eqref{eq:mh1}--\eqref{eq:mh2}. Now set $\Phi(v):=\phi$.
Suppose additionally that $(G(Lz,y),z^-)\leq 0$ for all $z\in W$ and $y\in  H^+$. {Hence,} $z(v)\geq 0$ and $\Phi(v)=Lz(v)\geq {\nu}$ for all  $v\in H$. {In addition, if $v_1\leq v_2$ implies} 
\begin{equation*}
(G(Lv,v_1)-G(Lw,v_2),(v-w)^+){\geq} 0,
\end{equation*}
for all $w,v$, then $z(v_1)\leq z(v_2)$ and $\Phi(v_1)\leq \Phi(v_2)$,
as $L$ is increasing. This finally shows that \eqref{eq:QVIforu}--\eqref{eq:mould} has the form \eqref{eq:QVI} {with $\Phi$ as an increasing operator} and $\mathbf{K}(v)$ given as
$\mathbf{K}(v):=\{v'\in V:v'\leq\Phi(v)\}$, where $\Phi(v)=\phi\in W$, 
and the pair $(z,\phi)$ is as given by \eqref{eq:mh1}--\eqref{eq:mh2}.

 Finally, assuming that $f\in\mathbf{A}_{\mathrm{ad}}=\{h\in V': 0\leq h\leq f_{\mathrm{max}}\}$, we have that the operators $\mathsf{m}\colon \mathbf{A}_{\mathrm{ad}}\to V$and $\mathsf{M}\colon \mathbf{A}_{\mathrm{ad}}\to V$ are well-defined: They map elements in $\mathbf{A}_{\mathrm{ad}}$ to minimal and maximal solutions to \eqref{eq:QVI} {in} the interval $[\underline{y},\overline{y}]=[0,A^{-1}f_{\mathrm{max}}].$

\subsubsection{A useful Mosco convergence result}

The obstacle operator $\Phi$ arising from \eqref{eq:QVIforu}-\eqref{eq:mould}, can be written as $\Phi(v)=C^{-1}(Lv)+\tilde{g}$ where $C$ is a (nonlinear) partial differential operator, and $\tilde{g}$ is some fixed element in $V$. In particular, this generates the need to consider Mosco convergence results when obstacles have specific structure. In this vein, we consider the following result.

\begin{theorem}\label{thm:MoscoQ}
	Let $\phi_n,\phi\in V$ for $n\in \mathbb{N}$. Suppose that $\phi_n\rightarrow \phi$ in $H$, and 
	\begin{equation*}
		\mathcal{Q} \phi_n\geq 0 \quad \text{ in } \: V \quad \text{ for all } \:n\in \mathbb{N},
	\end{equation*}
for some strongly monotone $\mathcal{Q}\in \mathscr{L}(V,V')$, such that $\langle \mathcal{Q}v^-, v^+\rangle\leq 0$ for all $v\in V$. Then,
	 \begin{equation*}
	 	\{w\in V: w \leq \phi_n \} \Mosco\{w\in V: w \leq \phi \} .
	 \end{equation*}
holds true.	 
\end{theorem}

\begin{proof}
	First note that since $H \hookrightarrow L^1(\Omega)$, then $\phi_n\to \phi$ in $H$ also implies strong convergence in $L^1(\Omega)$. It follows  by Proposition \ref{iiMosco} that \eqref{itm:2} in Definition \ref{definition:MoscoConvergence} holds true. In order to prove  \eqref{itm:1} in Definition \ref{definition:MoscoConvergence} we consider the following construction based on singular perturbations.
	
	Let $w\in V$ such that $w\leq \phi$ be arbitrary and let $w_n$ for $n\in \mathbb{N}$ be defined by
\begin{equation}\label{eq:SingPert}
\langle r_n\mathcal{Q}w_n+w_n, v\rangle = ( \tilde{w}_n, v),\text{ for all } v\in V,
\end{equation}
where $r_n:=\|\tilde{w}_n- w\|_H$ and $\tilde{w}_n:=\min(w, \phi_n)$, and note that $\tilde{w}_n\rightarrow w$ in $H$ and $w\in V$. Then, we can prove that $w_n\to w$ in $V$. Since  $\mathcal{Q}$ is linear, bounded, and $\langle \mathcal{Q}v, v\rangle\geq c\|v\|_V^2$ for all $v\in V$, from the definition of $w_n$ we have
\begin{align}\notag
	r_nc\|w_n-w\|_V^2+\|w_n-w\|_H^2&\leq \left\langle (r_n\mathcal{Q}+I)(w_n-w), w_n-w\right\rangle\\\label{eq:usesing}
	&\leq \langle \tilde{w}_n-w, w_n-w\rangle- r_n\langle\mathcal{Q}w, w_n-w\rangle\\\notag
	&\leq  r_n(C_p+\|\mathcal{Q}w\|_{V'})\| w_n-w\|_V,
\end{align}
where $C_p$ is the constant for the embedding $V \hookrightarrow H$. This implies that, $\{w_n\}$ is bounded in $V$, so that $w_n\rightharpoonup w^*$ (along a subsequence) for some $w^*\in V$. By taking the limit in \eqref{eq:SingPert}, it is shown that $w^*=w$ and that $w_n\rightharpoonup w$ in $V$ not only along a subsequence. It further follows that $w_n\to w$ in $H$, and since from \eqref{eq:usesing} we observe
\begin{align}
	r_nc\|w_n-w\|_V^2+\|w_n-w\|_H^2
	&\leq  r_n(\|w_n-w\|_H+\langle\mathcal{Q}w, w-w_n\rangle),
\end{align}
we have that $w_n\to w$ in $V$.

Next we prove that $w_n\leq \phi_n$. Consider $v=(w_n-\phi_n)^+$ and 
let us subtract $\left\langle r_n\mathcal{Q}\phi_n+\phi_n, v\right\rangle$ 
from both sides of \eqref{eq:SingPert}. Then, we get
\begin{align*}
 & r_n\left\langle\mathcal{Q}(w_n-\phi_n), (w_n-\phi_n)^+\right\rangle +\|(w_n-\phi_n)^+\|^2_H=\\
 &\quad -r_n\left\langle \mathcal{Q}\phi_n, (w_n-\phi_n)^+\right\rangle+( \min(w, \phi_n)-\phi_n, (w_n-\phi_n)^+).
\end{align*}
Note that $\min(w, \phi_n)-\phi_n\leq 0$ and by assumption $\mathcal{Q}\phi_n\geq 0$. Therefore the right hand side is less or equal to zero. Additionally, since $\mathcal{Q}$ is linear, $\langle \mathcal{Q}v^-, v^+\rangle\leq 0$, and $\langle \mathcal{Q}v, v\rangle\geq c\|v\|_V^2$ for all $v\in V$, we observe that
\begin{align*}
&r_n c\|(w_n-\phi_n)^+\|_V^2 +\|(w_n-\phi_n)^+\|^2_H\leq\\
&\qquad\le r_n\left\langle\mathcal{Q}(w_n-\phi_n)^+, (w_n-\phi_n)^+\right\rangle +\|(w_n-\phi_n)^+\|^2_H\leq 0.
\end{align*}
This yields $w_n\leq \phi_n$, i.e., \eqref{itm:1} in Definition \ref{definition:MoscoConvergence} holds true. This completes the proof.
\end{proof}

In view of the previous result, we assume throughout the rest of  this section the following continuity assumption on the obstacle map $\Phi$.

\vspace{.1cm}

\begin{assumption}\label{PhiAss}
If $v_n\rightharpoonup v$ in $V$, then $\Phi\colon H\to H\cap [\nu,+\infty)$ satisfies one of the following conditions:
\begin{itemize}
  \item[$\mathrm{(a)}$] $\Phi(v_n)\rightarrow \Phi(v)$ in $L^\infty(\Omega)$, or $\Phi(v_n)\rightarrow \Phi(v)$ in $V$.
 \item[$\mathrm{(a)}$] $\Phi(v_n)\rightarrow \Phi(v)$ in $H$ and if $v\in V\cap H^+$, then $\Phi(v)\in V$ and $\mathcal{Q} \Phi(v)\geq 0$ in $V$, for some strongly monotone $\mathcal{Q}\in \mathscr{L}(V,V')$, such that $\langle \mathcal{Q}v^-, v^+\rangle\leq 0$ for all $v\in V$.
\end{itemize}  
\end{assumption}

\vspace{.1cm}

Hence, by Proposition \ref{iMosco}, and Theorem \ref{thm:MoscoQ}, we assume that $\Phi$ satisfies conditions to guarantee Mosco convergence of the sets $\K(v_n)$, provided that $v_n\rightharpoonup v$ in $V$, to $\K(v)$.

\subsubsection{Perturbation of minimal and maximal solutions}

Existence of solutions to the QVI of interest is established if the following property holds 
 \begin{equation*}
 	v_n\rightharpoonup v \quad \text{ implies that } \quad  \K(v_n)\Mosco\K (v).
 \end{equation*}
 However, we are interested in the stability properties of the maps
\begin{equation*}
\mathbf{A}_{\mathrm{ad}}\ni f\mapsto \mathsf{m}(f) \qquad \text{and} \qquad \mathbf{A}_{\mathrm{ad}}\ni f\mapsto \mathsf{M}(f),
\end{equation*}
where
\begin{equation*}
	\mathbf{A}_{\mathrm{ad}}=\{g\in V': 0\leq g\leq f_{\mathrm{max}}\},
\end{equation*}
and hence additional assumptions are needed. In what follows, we establish our fundamental result concerning the behavior of the maps $f\mapsto \mathsf{m}(f)$ and $f\mapsto \mathsf{M}(f)$. As in the previous section we assume  that $[\underline{y},\overline{y}]=[0, A^{-1}f_{\mathrm{max}}]$.

\begin{thm}\label{thm:StabilityMinMax}
Let $\{f_n\}\subset L^\infty(\Omega)\cap \mathbf{A}_{\mathrm{ad}}$ be such that $f_n\geq c$ for some constant $c>0$, and
$\lim_{n\to\infty} f_n=f^*$ in $L^\infty(\Omega)$. 
Suppose that the upper bound mapping $\Phi$ satisfies Assumption \ref{PhiAss} { (page \pageref{PhiAss})},  and that 
\begin{equation}\label{eq:sqrt}
	\lambda  \Phi(y)\geq \Phi(\lambda  y) \quad \text{ for any } \quad \lambda>1,\, y\in H.
\end{equation}
Then the assertions
\begin{itemize}
\item[$\mathrm{(i)}$] { The sequence of minimal solutions satisfy } 
					\begin{equation}\label{eq:ConvergenceMin}
    \mathsf{m}(f_n)\rightarrow \mathsf{m}(f^*) \text { in } H, \qquad \text{and} \qquad  \mathsf{m}(f_n)\rightharpoonup \mathsf{m}(f^*) \text { in } V.
\end{equation}
\item[$\mathrm{(ii)}$] { The sequence of maximal solutions satisfy } 
					\begin{equation}\label{eq:ConvergenceMax}
    \mathsf{M}(f_n)\rightarrow \mathsf{M}(f^*) \text { in } H, \qquad 
    \text{and} \qquad  \mathsf{M}(f_n)\rightharpoonup \mathsf{M}(f^*) 
    \text { in } V.
\end{equation}
\end{itemize}
hold true.
\end{thm}

A few words are in order concerning the previous result. Note that if $\Phi$ satisfies Assumption \ref{PhiAss} { (page \pageref{PhiAss})}, but not necessarily \eqref{eq:sqrt}. Then, it is possible to prove that $\mathsf{m}(f_n)$ (and $\mathsf{M}(f_n)$) converge to solutions, elements of $\mathbf{Q}(f^*)$, but not necessarily to $\mathsf{m}(f^*)$ (and $\mathsf{M}(f^*)$). Assumption \eqref{eq:sqrt} provides the stability of extremal points. Structurally speaking, if $\Phi$ is a superposition operator, it states that $\Phi(x)\simeq (x^+)^{1/p}$ for some $p$.

\bibliography{DatabaseBibliography}{}
\bibliographystyle{abbrv}

\end{document}